\documentclass[11pt]{article}

\usepackage[a4paper, left=3.2cm,top=3.2cm]{geometry}
\usepackage{authblk}
\usepackage[onehalfspacing]{setspace}

\usepackage{hyperref}

\usepackage{amsmath,amsthm,amssymb}
\usepackage{graphicx}
\usepackage{epstopdf}
\usepackage{bold-extra}
\usepackage{siunitx}
\usepackage{bm}
\usepackage{concmath}
\usepackage[T1]{fontenc}
\usepackage{mathtools}
\usepackage{enumitem}
\usepackage{soul}
\usepackage{xcolor}
\setitemize{label=\scriptsize{$\blacksquare$},topsep=0em}
\setenumerate{label=(\alph*), topsep=0em}

\newcommand{\la}{\lambda}

\newcommand{\N}{\mathbb N}
\newcommand{\XX}{\mathbb X}
\newcommand{\YY}{\mathbb Y}

\newcommand{\s}{\mathbb S}

\newcommand{\Ao}{\mathbf{A}}
\newcommand{\To}{\mathbf{T}}
\newcommand{\Po}{\mathbf{P}}
\newcommand{\Ro}{\mathbf R}

\newcommand{\Bo}{\mathbf{B}}
\newcommand{\reg}{r}

\newcommand{\eps}{\epsilon}

\def\plus{{\boldsymbol{\texttt{+}}}}

\newcommand\norm[1]{\Vert#1\Vert}

\newcommand\set[1]{{\{#1\}}}

\DeclareMathOperator*{\argmin}{arg\,min}
\DeclareMathOperator{\id}{Id}
\DeclareMathOperator{\fix}{Fix}

\newcommand{\signal}{x}

\newcommand{\data}{y}

\newcommand{\noise}{\xi}

\usepackage{amsthm}

\newtheorem{theorem}{Theorem}
\newtheorem{definition}[theorem]{Definiton}

\newtheorem{proposition}[theorem]{Proposition}
\newtheorem{remark}[theorem]{Remark}

\numberwithin{equation}{section}
\numberwithin{theorem}{section}

\allowdisplaybreaks

\title{Superiorized Regularization of Inverse Problems}

\date{}

\author{Aviv Gibali}

\affil{Department of Mathematics, ORT Braude College\authorcr
Karmiel 2161002, Israel
 \authorcr E-mail:  \texttt{avivg@braude.ac.il}
 }

\author{Markus Haltmeier}

\affil{Department of Mathematics, University of Innsbruck\authorcr
Technikerstrasse 13, 6020 Innsbruck, Austria
 \authorcr E-mail:  \texttt{markus.haltmeier@uibk.ac.at}
 }

\begin{document}

\maketitle

\begin{abstract}
Inverse problems are characterized by their inherent non-uniqueness and  sensitivity with respect
to data perturbations. Their stable solution  requires the application of regularization methods including  variational and iterative  regularization methods. Superiorization  is a heuristic approach that can steer basic iterative algorithms to have small value of certain regularization functional while keeping the algorithms simplicity and computational efforts, but is able to account for additional prior information. In this note, we combine the superiorization methodology with iterative regularization methods and show that the superiorized version of the scheme yields  again a regularization method, however accounting for different prior information.

\medskip\noindent\textbf{keywords}
Inverse problems; iterative regularization; superiorization; generalized inverse.

\end{abstract}

\section{Introduction}
\label{sec:intro}

Throughout this  paper, let $\XX$ and $\YY$ be Hilbert
spaces and let  $\Ao \colon \XX \to \YY$
be a bounded linear operator.
We study the stable solution of the inverse problem of estimating
the unknown $\signal \in \XX$ from data
\begin{equation} \label{eq:ip}
	 \data_\delta =  \Ao (\signal) + \noise_\delta  \,.
\end{equation}
Here $\noise_\delta$ models the noise (data error) that we assume to satisfy
the noise estimate $\norm{\noise_\delta} \leq \delta$ with noise level
$\delta  \geq 0$.
In particular, we are interested in the ill-posed case, where  solutions  of \eqref{eq:ip} are non-unique (even for exact data), are unstable with respect to data perturbations, or suffer from both issues. Several practically relevant  applications can be formulated in such form, including computed tomography, geophysical  imaging, or  nondestructive testing \cite{engl1996regularization,scherzer2009variational,natterer01mathematical}.
In order to account for the ill-posedness of \eqref{eq:ip} one has to apply regularization  methods, which enforce uniqueness
by using a  suitable  right inverse of $\Ao$,  and stabilize the inversion by relaxing the exact solution concept.

Classical regularization methods use approximations of the Moore-Penrose inverse, which is the right inverse defined by selecting the solution of $\Ao \signal = \data$ with { minimal Hilbert space norm.}  Various variational and iterative regularization methods fit this into this framework. In particular, the Landweber method  \cite{engl1996regularization,landweber1951iteration} which uses the iterative update
\begin{equation} \label{eq:landweber}
	 \signal_{k+1}^\delta   = \signal_k^\delta - \la \Ao^*(\Ao \signal_k^\delta - \data^\delta ) \,,
\end{equation}
{ ($\Ao^*$ denotes the adjoint of $\Ao$)} with starting value $\signal_0^\delta = 0$ is probably the most established iterative regularization method
based on the Moore-Penrose inverse. The Landweber method can be seen as gradient based iteration applied to the
least squares functional $\frac{1}{2} \norm{\Ao \signal - \data^\delta}^2$ combined with a suitable stopping rule in the case of noisy data.
For exact data $\data \in \Ao(\XX)$, the Landweber  iteration converges to the { Moore-Penrose inverse applied  to exact data,  which is equal to the minimal norm solution} $\Ao^\plus(\data) =  \argmin_\signal \{ \norm{\signal} \mid  \Ao \signal = \data \}$. For noisy data, the  iterates $ \signal_k^\delta$  define a regularization method in sense that together with a suitable  stopping rule $k = \kappa(\delta)$ we have
$\norm{\signal_{\kappa(\delta)}^\delta - \Ao^\plus (\data) } \to 0$ as $\delta \to 0$.

Minimal Hilbert space norm solutions, however,  are often not the relevant ones in practical application. For example, signals of interest may be
characterized by having small values of some regularizer $\reg \colon \XX \to [0, \infty]$ such  as the $\ell^1$-norm with respect to
a certain  frame, or the total variation  
{ (see, for example,   \cite{acar1994analysis,calvetti2018bayes,candes2002recovering,grasmair2011necessary}).}

Generalizing regularization methods based on the Moore-Penrose inverse consider approximations of
$\reg$-minimizing solutions, that { are elements} in $\argmin_\signal \{ \reg(\signal) \mid  \Ao \signal = \data \}$. However, iterative  algorithms implementing
such approaches are typically more complex than gradient type methods such  as \eqref{eq:landweber}.

In this paper, we therefore follow a different approach, where instead of iterative algorithms aiming for strictly minimizing
$ \reg(\signal)$ over the set of all solution of $\Ao \signal = \data$, we modify  \eqref{eq:landweber} such that it reduces the value of the regularizer opposed to the basic iteration.
More precisely, we define
\begin{equation}\label{eq:super}
\begin{cases}
\signal_{k+1/2}^\delta=\signal_k^\delta + t_k \Phi_k(\signal_k^\delta)\\
\signal_{k+1}^\delta= \signal_{k+1/2}^\delta  -\la  \Ao^*(\Ao \signal_{k+1/2}^\delta  - \data^\delta )  \,,
\end{cases}
\end{equation}
where $\Phi_k(\,\cdot\,)$ are continuous perturbations with bounded range and $t_k$ is a summable sequence of non-negative numbers.
The perturbations are chosen such that: $\reg(\signal_{k+1/2}^\delta) \leq  \reg(\signal_k^\delta)$, meaning that the perturbed update has { no larger} value of $\reg$ than the unperturbed update.
In that context one refers  to \eqref{eq:super} as the superiorized version of \eqref{eq:landweber}.

The concept of superiorization is first introduced (under a different name) in \cite{butnariu2007stable} and studied further by many authors is various fields and applications. A main ingredient in these works is the perturbation resilience in the case of exact data \cite{davidi2009perturbation,censor2019derivative}. For the state of current research on superiorization one can check the website \cite{CensorSupPage}. In particular, see \cite{censor2019superiorization,helou2018superiorization,GuoCui18,GuoZhao19,zibetti2018total}
for some recent papers on superiorization in the context of least squares minimization including \eqref{eq:landweber}.

The aim of this paper  is to analyze superiorization for ill-posed problems and to study its convergence behavior as $\delta \to 0$.
As we shall show, the iteration \eqref{eq:super} again defines a regularization method, which however does not approximate the Moore-Penrose inverse but a more general right inverse, that can be adapted desired properties of solutions of $\Ao \signal = \data$.
To the best of our knowledge superiorization has not been previously analyzed as regularization method. We have selected the Landweber method as prime iterative regularization method but are convinced that many other superiorized iterative algorithms can be shown to yield regularization methods.


\section{Preliminaries}\label{sec:preliminaries}


In this section, we recall the concepts of regularization methods,
perturbation resilience and the superiorization methodology in
the Hilbert space context. Moreover, we derive some auxiliary results
that we require for later purpose.

\subsection{Regularization methods}

Let $\Ro \colon \Ao(\XX) \subseteq \YY \to \XX$  be a right inverse of $\Ao$, which means any  possibly non-linear mapping  such that $\Ao \circ \Ro (\data) = \data$ for all $\data \in \Ao(\XX)$.

For exact data,  the selection of a particular right inverse defines a unique solution concept for
the equation $\Ao \signal = \data$.
However, as the following Proposition~\ref{prop:ill} states, in the case that $\Ao$ has non-closed range, then any right inverse is unbounded. For example,   any compact operator and has  non-closed range including integral operators arising in  typical inverse problems.
This is the reason for using  regularizations methods as stable approximations of right inverses.

\begin{proposition} \label{prop:ill}
If $\Ao(\XX)$ is non-closed, then $\Ro$ is discontinuous.
\end{proposition}

\begin{proof}
Every solution the equation   $\Ao\signal = \data$ has  the form $\Ao^\plus (\data) + z$ where $z \in \ker(\Ao)$. In particular, we have $\Po_{\ker(\Ao)^\bot }\Ro = \Ao^\plus$. This shows that the the continuity   of $\Ro$ is   implies the continuity of $\Ao^\plus$. However, the continuity of $\Ao^\plus$
 implies  the closeness  of $\Ao(\XX)$ (see for example \cite{engl1996regularization}) and concludes the proof.
\end{proof}

Hence in the  ill-posed setting where $\Ao$ has non-closed range,  any
right inverse is discontinuous. For the stable solution of such inverse problems
one has to  apply regularization methods that are defined as follows.

\begin{definition}[Regularization method]\label{def:regmeth1}
A family  $(\Ro_k)_{k>0}$ of continuous operators $\Ro_k\colon \YY\rightarrow \XX$
is called  regularization of $\Ro$ if for all $\data \in \Ao(\XX)$
there exists a mapping $\kappa \colon (0,\infty)  \rightarrow \N$   with
$\kappa(\delta) \to \infty$ as $\delta \to 0$ and
\begin{equation}
\lim_{\delta\rightarrow 0} \sup \Bigl\{\| \Ro (\data) -\Ro_{\kappa(\delta)}(\data^\delta) \| \mid \data^\delta \in \YY \wedge  \|\data^\delta-\Ao \signal\|\leq \delta\Bigr\} = 0 \,.
\end{equation}
In this case we call $\kappa$ (a-priori) parameter choice rule  and  the pair $((\Ro_k)_{k \in \N},\kappa)$  a regularization method for the solution of $\Ao\signal = \data$.
\end{definition}

Classical regularization methods are adapted to the Moore-Penrose inverse $\Ao^\plus$  where $\Ao^\plus(\data)$ for $\data \in \Ao (\XX)$ is defined as the unique solution  of
$\Ao \signal = \data$ with minimal norm.  This  includes  classical Tikhonov regularization and classical iterative regularization methods. In particular, the Landweber iteration
\eqref{eq:landweber}  together  with a suitable stopping rule is known to be a { regularization method}.  In this work we generalize these results by integrating the
superiorization to the Landweber iteration and thereby adapting
to more flexible right inverses.
The following lemma gives a useful  guideline for creating regularization methods.

 \begin{proposition}[Point-wise approximations are  regularizations]\label{prop:point}
Let $(\Ro_k)_{k>0}$ be a family of continuous operators $\Ro_k\colon \YY\rightarrow \XX$ that converge point-wise to $\Ro$ on $\Ao(\XX)$. Then $(\Ro_k)_{k>0}$ is a regularization
of  $\Ro$.
\end{proposition}

\begin{proof}
We follow the proof given in \cite{engl1996regularization} for the special case $\Ro = \Ao^\plus$.
Let $\data \in \Ao(\XX)$. For  any $\eps > 0$ choose
$k(\eps)$ such    $\norm{\Ro_{k(\eps)}(\data) - \Ro(\data)} \leq \eps/2$.
Moreover, choose $\tau(\eps)$ such that for all $z \in \YY$ with
$\norm{\data - z}  \leq  \tau(\eps)$ we have
$\norm{ { \Ro_{k(\eps)}(\data)} - \Ro_{k(\eps)}(z)} \leq \eps/2$.
Without loss of generality we can assume that $\tau(\eps)$ is strictly increasing and continuous with
$\tau(0+)=0$.  We define  $\kappa \coloneqq k \circ \tau^{-1}$. Then, for every $\delta >0$
and $\norm{\data - \data^\delta}  \leq  \delta$ we have
\begin{align*} 
\norm{\Ro_{\kappa(\delta)}(\data^\delta) - \Ro(\data)}
&\leq
\norm{\Ro_{\kappa(\delta)}(\data) - \Ro(\data)}
+
\norm{\Ro_{\kappa(\delta)}(\data) - \Ro_{\kappa(\delta)}(\data^\delta)} \\
&=
\norm{\Ro_{k \circ \tau^{-1}(\delta)}(\data) - \Ro(\data)}
+
\norm{\Ro_{k \circ \tau^{-1}(\delta)}(\data) - \Ro_{k \circ \tau^{-1}(\delta)}(\data^\delta)} \\
&\leq
{\tau^{-1}(\delta)}/{2} + {\tau^{-1}(\delta)}/{2}  = \tau^{-1}(\delta)\,.
\end{align*}
Because $\tau^{-1}(\delta) \to 0$ as $\delta \to 0$ this completes the proof.
\end{proof}

Proposition~\ref{prop:point} will be used to show that superiorized Landweber method
is a regularization method.

\subsection{Superiorization}

Let { $\s \subseteq \XX$ be a given set, $\To \colon \XX \to \XX$ be an algorithmic operator for the problem of finding  elements in $\s$, and suppose the basic iteration
 \begin{equation} \label{eq:basic}
  \signal_{k+1}  = \To(\signal_k)
  \end{equation}
converges for all  $x_0 \in \XX$ to some element in $\s$.} The superiorization methodology  modifies
\eqref{eq:basic} such  that the resulting iteration still converges, however with a limit having
{ lower or equal value} of some functional $\reg \colon \XX \to [0, \infty]$. Superiorization comes with several  benefits. First, limits having small value of $\reg$ are often closer to elements   than the original limits of  \eqref{eq:basic}. Second, one can make use of existing algorithms in the form of \eqref{eq:basic}. Third, opposed to optimization problems that strictly  minimize the functional $\reg$, superiorized algorithms are often simpler and more efficient to implement.  In this context we also mention that exactly minimizing  the regularizer $\reg$ is  anyway not strictly required in many inverse problems because   the selection of the regularizer itself is often somehow heuristic.\\

The superiorized version of   \eqref{eq:basic} is defined as follows.

\begin{definition}\label{def:super}
Let $\s \subseteq \XX$, let  \eqref{eq:basic}
 for all  $x_0 \in \XX$ converge to an element in $\s$, and let
$\reg \colon \XX \to [0,\infty]$ be  convex and subdifferentiable.
One calls  $\signal_{k+1}  = \To(\signal_k + t_k d_k)  $,
where $t_k \geq 0$ with $\sum_{k\in \N} t_k < \infty $, and
\begin{equation} \label{eq:dk}
d_k :=
\begin{cases}
- {D_k}/{\norm{D_k}} & \text{ if } D_k \neq 0
\\
0 & \text{ otherwise\,, }
\end{cases}
\end{equation}
with $D_k \in  \partial \reg (x_k)$,
the  superiorized version of \eqref{eq:basic}.
\end{definition}

In Definition \ref{def:super} we used the superiorization strategy based on the subgradient. This approach can be generalized by using further superiorization strategies, such as derivatives free techniques, see { also \cite{censor2019superiorization}} in the context of least squares minimization. { Note that for the presented convergence analysis we assume  that perturbations having the  form $d_k  = \Phi_k (x_k)$ with  continuous  $\Phi_k(\, \cdot\,)$. This can be achieved  
 by  smoothing  the normalization  procedure in \eqref{eq:dk} 
 around zero.}\\

General questions concerning the superiorization methodology are the following:
\begin{enumerate}[label=(\alph*)]
\item\label{s1} Does  the superiorized iteration  converge?
\item\label{s2} Is the  limit $\hat \signal$ contained  $\s$?
\item\label{s3} Is  $\reg(\hat \signal)$ { not larger than}  $\reg$ evaluated at limit of the basic sequence?\\
\end{enumerate}

{ 
\noindent The underlying concept addressing  the  issues 
\ref{s1}, \ref{s2}  is perturbation resilience that is defined next. 
That limit points satisfy \ref{s3}  is referred  to as the  guarantee problem of the superiorization methodology  and could not be proven until today.}

\begin{definition}
Let  $\s \subseteq \XX$ and suppose { that for all initial values $x_0 \in \XX$ the  sequence \eqref{eq:basic} converges} to some element in $\s$.
Iteration \eqref{eq:basic} is called bounded perturbation resilient (with respect to $\s$), if for all  $x_0 \in \XX$, all   $(t_k)_{k\in \N}  \in [0, \infty)^\N$ with $\sum_{k\in \N} t_k < \infty$ and  all bounded sequences $(d_k)_{k\in \N} \in \XX^\N$, the perturbed iteration
 \begin{equation} \label{eq:super}
 	\signal_{k+1}  = \To(\signal_k + t_k d_k )
\end{equation}
converges to some element  in $\s$.
\end{definition}

The following result which is a direct consequence  of \cite[Theorem~5]{butnariu2006convergence}  shows that the first two questions in the above list can be positively answered for a wide class of basic iterations.

\begin{proposition} \label{prop:perturbed}
Let $\To$ be non-expansive and all basic iterates \eqref{eq:basic} converge strongly to some
element $\fix(\To) \neq \emptyset$. Then \eqref{eq:basic} is bounded perturbation resilient.
\end{proposition}

\begin{proof}
Under the given assumptions,      \cite[Theorem~5]{butnariu2006convergence} states that
any  sequence  $(z_k)_{k \in \N}$ satisfying
$\sum_{k \in \N} \norm{z_{k+1} - \To(z_k )}<  \infty$ converges to a fixed point of $\To$.
Because of the  non-expansiveness of $\To$, the iteration \eqref{eq:super} satisfies
$\sum_{k \in \N}  \norm{x_{k+1} - \To(x_k )} \leq   \sup_k \norm{d_k} \sum_{k \in \N}  t_k$ which allows applying above mentioned result.
\end{proof}

We will apply Proposition \ref{prop:perturbed} to show the perturbation resilience of the Landweber iteration for exact data. In that context, we will also discuss regularizing properties of the superiorized Landweber iteration for noisy data, which, to the best of our knowledge, has not been investigated so far.

%
%

\section{Convergence analysis}
\label{sec:conv}

Recall that   $\Ao \colon \XX \to \YY$ is a  bounded linear  operator and choose  $\lambda \in  (0, 1/\norm{\Ao}^2)$.
Moreover, let  $\Phi_k\colon \XX \to \XX$ be a family of  continuous mappings with bounded range
and let $(t_k)_{k \in \N}$ be a summable sequence of nonnegative numbers.

{ We investigate the perturbed Landweber iteration for possibly noisy data that is defined as follows:
\begin{equation}\label{eq:super-landweber}
\begin{cases}
\signal_{k+1/2}^\delta = \signal_k^\delta  +  t_k \Phi_k(\signal_k^\delta)\\
\signal_{k+1}^\delta    =\signal_{k+1/2}^\delta - \la \Ao^*(\Ao \signal_{k+1/2}^\delta - \data^\delta) \,,
\end{cases}
\end{equation}
with initial data $\signal_0^\delta = 0$.} For properties of the Landweber operator the readers are referred to the work of \cite{cegielski14}. The index $\delta >0$ stands for the noise level and the given data $\data^\delta$ satisfy the estimate $\norm{\data - \data^\delta} \leq \delta$ with $\data \in \Ao(\XX)$. To indicate the dependence  of the iterates on the given data we write $\Bo_k(\data^\delta) \coloneqq \signal_k^\delta $.

\subsection{Exact data}
\label{sec:exact}

Our aim is to show that \eqref{eq:super-landweber} defines a regularization method. We first start with the convergence in the exact data case.

\begin{theorem}[Convergence for exact data]\label{thm:exact}
The Landweber iteration \eqref{eq:landweber} is perturbation resilient. That is, for  all
  $x_0 \in \XX$, the  perturbed Landweber iteration with exact data \eqref{eq:super-landweber} converges to a solution of  the equation $\Ao(\signal) = \data$. In particular, the limits of the iteration \eqref{eq:super-landweber} define a right inverse  $\Bo \colon \Ao(\XX) \to \YY \colon \data \mapsto  \lim_{k \to \infty} \Bo_k(\data) $ of $\Ao$.
\end{theorem}

\begin{proof}
The basic Landweber iteration  is known to strongly converge to a solution
 of $\Ao(\signal) = \data$ (see for { example \cite[Theorem 6.1]{engl1996regularization}}). Moreover, note that \eqref{eq:super-landweber} is a perturbed fixed point iteration with the operator
 $ \To(\signal) = (\id - \la \Ao^* \Ao) (\signal) + \Ao^*\data$.  We have $ \norm{\To(\signal)-\To( \signal_0)} \leq  \norm{\id - \la \Ao^* \Ao} \norm{\signal-\signal_0} \leq \norm{\signal-\signal_0}$. Hence $\To$ is non-expansive and because $\data \in \Ao(\XX)$, we have
 $\fix(\To) = \set{\signal \mid \Ao \signal = \data}$. Therefore  the claim follows from Proposition
 \ref{prop:perturbed}.
\end{proof}

%

\subsection{Noisy  data}
\label{sec:noisy}

Proposition~\ref{prop:ill} implies that in the ill-posed case where $\Ao(\XX)$ is non-closed,
the right inverse $\Bo$ { defined by} the perturbed Landweber exact data iteration \eqref{eq:super-landweber}
is discontinuous. Therefore, it has to be regularized. Following the  iterative regularization strategy,
the  regularization we use in this paper comes from early stopping the  noisy data iteration.
Recall that  we write $\Bo_k(\data^\delta) = \signal_k^\delta$  for the iterates defined in \eqref{eq:super-landweber}, defining mappings $\Bo_k\colon \YY \to \XX$.

\begin{theorem}[Convergence for noisy data]\label{thm:noisy}
For  all $\data \in \Ao(\XX)$  there exists a parameter choice rule $\kappa \colon (0, \infty) \to \N$ such that the pair  $((\Bo_k)_{k\in \N}, \kappa) $
is a regularization  method for the solution of  $\Ao \signal =  \data$ adapted to $\Bo$.
In particular,
\begin{equation}\label{eq:conv-nioisy}
\lim_{\delta \to 0}    \norm{\Bo (\data) - \signal_{\kappa(\delta)}^\delta }
= 0
\end{equation}
 and all  families of noisy data $(\data^\delta)_{\delta>0} $
with  $\norm{\data-\data^\delta} \leq \delta$.
Moreover,  the truncated iterates  form a regularization $(\Bo_k)_{k \in \N}$ of $\Bo$.
\end{theorem}

\begin{proof}
For exact data, according to Theorem~\ref{thm:exact}, we have $\lim_{k \to \infty} \Bo_k(\data)
= \Bo(\data)$ pointwise
for  $\data \in \Ao(\XX)$. Moreover, in an inductive manner one verfies that the mappings   $\Bo_k \colon \YY \to \XX$ are  continuous for all $k \in \N$. Proposition~\ref{prop:point} therefore implies the existence  of
$\kappa \colon (0, \infty) \to \N$ such that the pair $((\Bo_k)_{k\in \N}, \kappa) $
is a regularization  method for the solution of  $\Ao \signal =  \data$ adapted to $\Bo$.
In particular, \eqref{eq:conv-nioisy}  holds and $(\Bo_k)_{k \in \N}$ is a regularization of $\Bo$.
\end{proof}

\begin{remark}
Theorem \ref{thm:noisy} shows that there exists a parameter choice rule which yield a regularization method, but this is not given explicitly. In order to define this parameter choice rule, further investigation and additional prior information is needed; for example one can explore the following strategies.
\begin{itemize}
  \item Estimate the Lipschitz  constant of $B_k$ to find a-priori rules
  $k = \kappa(\delta)$.
  \item Show that the discrepancy principle gives a-posteriori rule $k = \kappa(\delta,y^\delta)$.
\end{itemize}
For the second strategy, relations between the discrepancy principle  and the concept of
strong perturbation resilience might be useful~\cite{herman2012superiorization}.
\end{remark}

\section{Conclusion}
\label{sec:conclusion}

In this paper we showed that superiorization concept applied to the Landweber method gives a regularization method for the solution of inverse problems. Basically, our main result states that truncating the superiorized (or perturbed) Landweber iteration depending on the noise level, is stable and convergent in the limit $\delta \to 0$.
To the best of our knowledge, such regularization properties have not been investigated previously for the superiorization methodology.

However, many relevant questions following this work regarding superiorization and regularization remain open and call for further investigations.

\begin{enumerate}[parsep=0em, label=\arabic*:]
\item
Explicit  parameter choice rules must  be derived.

\item
For the Landweber (and many  related  regularization techniques), the discrepancy principle which chooses the first index with $\norm{\Ao \signal_k^\delta - \data^\delta} \leq  c \delta$ yields an
 admissible parameter choice. It is unclear if a similar parameter choice  rule for the superiorized Landweber iteration exists.

\item
The Landweber  method is often quite slow. Accelerated iterative regularization might be investigated in combination with the superiorization technique.

\item
Convergence rates (quantitative estimates between the exact solution $\Bo \data$ for exact data and  regularized solutions $\Bo_k \data^\delta$ for noisy data) are well established for the Landweber iteration. Deriving such rates for the superiorized version  seems  a difficult issue.
In that context characterizing  $\Bo \data$ might be useful.

\item  Superiorization in the context of nonlinear inverse problems is another aspect that we have not touched.
\end{enumerate}

Investigating such issues are interesting lines of future research.


\section*{Acknowledgements}
The work of M. Haltmeier has been supported by the Austrian Science Fund (FWF), project P 30747-N32.

\end{document}